\documentclass[11pt]{amsart}
\usepackage[T1]{fontenc}
\usepackage[utf8]{inputenc}

\usepackage{amsmath,amsthm,amsfonts,amssymb,latexsym,hyperref,color}
\usepackage{fullpage,amscd,version}

\newtheorem{theorem}{Theorem}[section]

\newtheorem{proposition}[theorem]{Proposition}
\newtheorem{lemma}[theorem]{Lemma}
\newtheorem{remark}[theorem]{Remark}

\renewcommand{\d}{\mathrm d}
\renewcommand{\P}{\mathbb P}

\newcommand{\E}{\ensuremath{\mathbb{E}}}
\newcommand{\N}{\mathbb N}

\newcommand{\Z}{\mathbb{Z}}
\newcommand{\T}{\mathbb{T}}
\newcommand{\cC}{\mathcal{C}}
\newcommand{\cF}{\mathcal{F}}
\newcommand{\cI}{\mathcal{I}}
\newcommand{\cP}{\mathcal{P}}
\newcommand{\cN}{\mathcal{N}}
\newcommand{\cR}{\mathcal{R}}
\newcommand{\cS}{\mathcal{S}}
\newcommand{\cU}{\mathcal{U}}
\newcommand{\id}{1\hspace{-0,9ex}1}
\newcommand{\bfxi}{\boldsymbol{\xi}}
\newcommand{\e}{\mathrm e}

\title{Renewal contact process with dormancy}

\date{March 5, 2025}

\begin{document}

\maketitle 

\centerline{{\sc Noemi Kurt}\footnote{Goethe-Universität Frankfurt, Robert-Mayer-Straße 10, 60325 Frankfurt am Main, Germany, {\tt kurt@math.uni-frankfurt.de}}, {\sc Michel Reitmeier}\footnote{Goethe-Universität Frankfurt, Robert-Mayer-Straße 10, 60325 Frankfurt am Main, Germany, {\tt reitmeier@math.uni-frankfurt.de}}, {\sc Andr\' as T\'obi\'as}\footnote{Department of Computer Science and Information Theory, Budapest University of Technology and Economics, Műegyetem rkp. 3., H-1111 Budapest, Hungary, and HUN-REN Alfréd Rényi Institute of Mathematics, Reáltanoda utca 13--15, H-1053 Budapest, Hungary, {\tt tobias@cs.bme.hu}}}

\begin{abstract}
We consider the contact process with dormancy, where wake-up times follow a renewal process. Without infection between dormant individuals, we show that the process under certain conditions grows at most logarithmically. On the other hand, if infections between dormant individuals are possible, the process survives with positive probability even on finite graphs.
\end{abstract}

\noindent \emph{Keywords and phrases.} Contact process, renewal process, dormancy, percolation. \\
\emph{MSC 2020.} Primary 60K06, 60K35; secondary 82B43, 92D15.

\medskip

\section{Introduction and motivation}
Dormancy, referring to the ability of individuals to enter a reversible state of inactivity, is a widespread phenomenon in biological systems ranging from population genetics, epidemiology, ecology to neuroscience and oncology \cite{LdHWB21}. Its mathematical modelling has received quite some attention in the probabilistic community over the last years. However, dormancy phenomena are not restricted to biological systems. For example in \cite{Floreani}, particle systems with switching are considered, leading to interesting and sometimes surprising phenomena. The contact process with switching, with dormancy as a special case, has been introduced in \cite{BlathHermannReitmeier}. It can be seen as a toy model to investigate effects of (host) dormancy on the spread of an infection. The contact process with switching is related to the contact process in certain dynamical random environments, see \cite{LinkerRemenik, SeilerSturm}. From a mathematical perspective, particle systems with dormancy also provide an accessible framework to investigate the interplay of spatial structure with an additional temporal structure that is created by (potentially very long) periods where particles are dormant. 

\medskip

The contact process with dormancy of \cite{BlathHermannReitmeier} is a contact process in a dynamical environment, which in their case consists of an additional state being assigned to each site, that may be either active or dormant. The environment evolves autonomously and independently at each site, while the rates of the contact process depend on the activity state at the site and its neighbours. In the present paper, we adapt the contact process with dormancy  to accommodate dormancy periods with heavy tails. On the technical side, we can build on the approach in a series of papers \cite{RCP1,RCP2,RCP3,RCP_finite,M. Hilario,RCPnew} that investigate the contact process with renewals but without dormancy.
In these papers, recoveries in the graphical construction are not governed by a Poisson process, but by a renewal process with interarrival times distributed according to some law $\mu$ on $[0,\infty).$ An interesting choice for $\mu$ are power laws. Among other results, the authors prove that under some tail conditions on $\mu$ (which in particular imply that there exists no first moment) the critical value $\lambda_c$ of the infection rate for this contact process is 0 \cite{RCP1}. \cite[Section 5]{RCP3} provides an example where the same holds even though $\mu$ is attracted to a stable law of index 1. On the other hand, if any moment of order strictly larger than 1 exists, then $\lambda_c>0$, see \cite{RCP2, RCP3}. Note that \emph{a priori}, this does not necessarily imply that $\lambda_c<\infty$ also holds, i.e.\ that there is a nontrivial phase transition. But at least for certain classes of distributions $\mu$, it has been shown that that is indeed the case~\cite{RCP3,RCPnew}.

\medskip
The aim of the present paper is to include renewals into the contact process with dormancy. More precisely, renewals will be included in the switches from the dormant to the active state, whereas the infections and recoveries will have the usual Markovian structure with exponential waiting times. The recovery rates and the infection rates then depend on the activity state of individuals and (in the case of infection rates) their neighbours, while the activity state (active/dormant) of an individual
evolves autonomously from the infection and independently of other individuals. Depending on the parameter choices, we can find regimes with quite different behaviour from the classical contact process. In fact, we find that if no infection happens between dormant individuals and if the dormancy periods approximately have power-law tails of exponent $\alpha \in (0,1)$,
the infection grows at most logarithmically. This strengthens a result of \cite{M. Hilario}, where at most sub-linear growth was found in a related model. Our argument relies on the Dynkin--Lamperti Theorem.  On the other hand, if infections between dormant individuals are possible, we can show that the process survives with positive probability 
even on finite graphs under suitable conditions.\\

To put the results into perspective, we recall that the classical contact process displays linear growth in the survival regime, see \cite{Durrett}. This is also true at least under certain conditions in the contact process with dormancy of \cite{BlathHermannReitmeier}. Indeed, the heavy tailed dormancy periods are responsible for the slowdown of the infection, since they lead to a blocking effect on a long time scale. In particular, our result does not apply to the Markovian case \cite{BlathHermannReitmeier} or to related Markovian models with blocking sites \cite{LinkerRemenik,Velasco}. 

\medskip
While our general strategies follow \cite{RCP1,RCP_finite} and \cite{M. Hilario}, the additional switches between active and dormant add another layer of complexity. Moreover, we are not aware of any coupling between our model and the ones in  \cite{RCP1,RCP_finite,M. Hilario} due to the fact that we cannot couple a random subset of Poisson points (i.e.\ recovery/infection symbols) occurring while the environment is dormant with a non-trivial renewal process.
We build the model on the usual graphical representation, however, now the events in this representation depend on the activity state. Therefore, we need to apply particular care when constructing suitable grids or events for the percolation arguments, which often require additional structure.

\section{Model and main results}
At first, we recall that a general renewal process $\cS=\{S_n\}_{n\geq 0}$ with interarrival distribution $\mu$ on $[0,\infty)$ is a set of renewal points satisfying $S_n:=T_1+\ldots+ T_n$, where the interarrival times $T_i$ are independent with law $\mu$. For $\cS$ we introduce the \textit{current time} $C^{\cS}(t)$ and the \textit{excess time} $E^{\cS}(t)$ by
\begin{equation*}
C^{\cS}(t):=t-S_{N(t)}\quad\text{and}\quad E^{\cS}(t):=S_{N(t)+1}-t
\end{equation*}
where $N(t):=\sup\{n\geq0: S_n\leq t\}$ is the number of renewals up to time $t$.

\medskip
To define our model let $G=(V,E)$ be a connected graph of bounded degree. Usually, we will consider the Euclidean lattice $(\Z^d,\E^d)$, where $\E^d=\{\{x,y\}:|x-y|_1=1\}$ is the set of nearest neighbour edges.
Our process $(\bfxi_{t})_{t\geq 0}$ will have state space $F^V:=(\{0,1\}\times \{a,d\})^{V},$ where $0$ stands for healthy, $1$ for infected, $a$ for active and $d$ for dormant. Thus, the tuple $\bfxi_t(x)\in F$ describes the local state of site $x$ at time $t$. We construct our process via the usual graphical representation of  Harris \cite{H78}. That is, for each vertex in $V$ we have a Poisson process with rate $\delta>0$ which is the recovery rate for the active individuals. We assume that no recovery is possible in the dormant state. For each edge we have Poisson processes with rates $\lambda_{aa}, \lambda_{ad},\lambda_{da}, \lambda_{dd} \geq 0$ which are the infection rates between the individuals of different types. The switches from dormant to active for every vertex are given by a renewal process with interarrival distribution $\varrho$ which has heavy tails and satisfies some regularity assumptions we specify later. Moreover, the switches from active to dormant happen according to a Poisson process with rate $\sigma>0$. To be more precise, we assume that we have the following independent processes: 
\begin{itemize}
	\item Recovery events: Poisson processes $\{\cR_{x}\}, x\in V, $ with rate $\delta$.
	\item Infections: Poisson processes $\{\cI_{e, \tau}\},~e\in E,
	\tau \in \{a,d\}^2$, with rate $\lambda_\tau\geq 0$. 
	\item Switches from active to dormant: Poisson processes $\{\cU_{x}\}, x\in V$, with rate $\sigma$.
	\item Switches from dormant to active: Renewal processes $\{\mathcal{S}_x\},x\in V$ with interarrival distribution $\varrho$.
\end{itemize} 
We denote the current and excess times of these renewal (Poisson) processes by $C_x^\delta,C_e^{\lambda_\tau},C_x^\sigma,C_x^\varrho$ and $E_x^\delta,E_e^{\lambda_\tau},E_x^\sigma,E_x^\varrho$, respectively.
Given these processes, the \emph{contact process with renewal dormancy} is constructed in the usual way, given some initial state $\xi\in F^V$. First, since the activity state for every site evolves autonomously from the infection and independent from all other sites, we define
\begin{equation*}
    \bfxi_t(x)_2=a\quad\text{if and only if }\quad C_x^{\varrho}(t)<C_x^{\sigma}(t),~\text{or}~C_x^{\varrho}(t)=C_x^{\sigma}(t)=t~\text{and}~\xi(x)_2=a,
\end{equation*}
where $\bfxi_t(x)_i$ denotes the $i$th component of $\bfxi_t(x)$. Note, the first condition says that the last switch for the activity of $x$ was from $d$ to $a$ and the second one covers the case where $x$ is initially active and no activity update occurs for $x$ until time $t$. Second, for $s<t$ and $x,y\in V$ we define (analogously to~\cite[Section 2.1]{BlathHermannReitmeier}) 
an \textit{infection path} from $(x,s)$ to $(y,t)$ as a c\`adl\`ag function $\gamma:[s,t]\to V$ with $\gamma_s=x$ and $\gamma_t=y$ such that:
\begin{itemize}
	\item The path can only jump to an adjacent vertex if there is an infection matching the activity types of the involved vertices, i.e.\ If $v=\gamma_{u-}\neq w=\gamma_{u}$ for some $u\in [s,t]$ then $e=\{v,w\}\in E$ and $u\in \cI_{e,\tau}$ with $\tau=(\bfxi_{u}(v)_2, \bfxi_{u}(w)_2)$.
	\item The path is not allowed to pass recovery symbols while it stays on an active site, i.e.\  If $u\in \cR_{\gamma_u}$ for some $u\in [s,t]$ then $\bfxi_u(\gamma_u)_2=d$.
\end{itemize}
By $(x,s)\rightsquigarrow(y,t)$ we denote the event that there exists an infection path from $(x,s)$ to $(y,t)$. Finally, the infected sites at time $t$ are just the ones reached by some infection path from an initially infected site:
\begin{equation*}
    \bfxi_t(x)_1=1\quad\text{if and only if we have }(y,0)\rightsquigarrow(x,t)~\text{for some }y\text{ with }\xi(y)_1=1.
\end{equation*}
The case when the processes $\{ \mathcal S_x \}$, $x \in V$ are also Poisson processes, i.e.\ the interarrival distribution $\varrho$ is exponential, was studied in~\cite{BlathHermannReitmeier}. 

\medskip
If we want to specify the initial configuration $\xi$ we write it in the superscript $\bfxi_t^\xi$. Often it will be convenient to describe the initial configuration $\xi$ by two sets $I,A\subset V$, which describe the set of initially infected, respectively active sites. Therefore we also write $\bfxi_t^{I,A}$.
Moreover, the graphical construction implies a universal coupling for all initial configurations which yield the same monotonicity, additivity and attractivity properties stated in \cite[Theorem 2.2]{BlathHermannReitmeier} for the Markovian set-up without renewals. It should be noted that this monotonicity holds with respect to the order $\preceq$ on $\{0,1\}\times\{a,d\}$ defined through \begin{equation*}
    (x_1,x_2)\preceq(y_1,y_2)\quad\text{if and only if}\quad x_1\leq y_1~\text{and}~x_2=y_2.
\end{equation*}In particular, the order $\preceq$ is not a total order on $\{0,1\}\times\{a,d\}$, and two configurations $\eta,\xi\in F^V$ can only be compared if the activity states for all sites are identical.

\medskip

Let $\P$ be the joint law of all renewal and Poisson processes we need for our construction on some suitable common probability space $\Omega$. Later it will be useful to consider it as a product measure $\P=\P^\varrho\times \P'$ on $\Omega^\varrho\times \Omega'$, where $\P^\varrho$ is the joint law of the renewal processes $(\cS_x)_{x\in V}$ and $\P'$ is the joint law over all other point processes in the construction. One question of interest is the survival of the infection, hence we define the extinction time as
\begin{equation*}
\tau^{I,A}:=\inf\{t\geq0:\bfxi_t^{I,A}(x)_1=0~\text{for all}~x\in V\}.
\end{equation*}
We usually start with all sites active and only one infected, since in view of the heuristics that most individuals are dormant in the long run (due to long dormancy periods, which are bound to occur under our conditions), this is the extreme case (in terms of extinction). Therefore, we use the shorthand notation $\tau^x:=\tau^{\{x\},V}$ and $\bfxi_t^x:=\bfxi_t^{\{x\},V}$.
We say the contact process survives if the survival probability is positive, i.e.\ $\P(\tau^x=\infty)>0$ for some $x\in V$.
In the case when $G$ equals $\mathbb Z^d$ with the nearest-neighbour edges, apart from survival we are also interested in the growth of the process and therefore define the range of the process as
\begin{equation*}
    r_t:=\max\{\Vert x \Vert_\infty  \colon \bfxi_t^{0}(x)_1=1\}.
\end{equation*}
We state three assumptions regarding the measure $\varrho$ which are formulated with respect to the corresponding renewal process $\cS$. The first one (G) originates from~\cite{RCP1}, while the last two (S) and (S$^*$) come from~\cite{RCP_finite}.
\begin{enumerate}
	\item[(G)] (Gap condition.) There exist $t^*>0$ and $\varepsilon^*>0$ such that
	$$\P(\mathcal{S}\cap[t,t+t^{\varepsilon^*}]\neq \emptyset)\leq t^{-\varepsilon^*}\quad\text{ for all }t\geq t^*.$$
	\item[(S)] 
    The law $\varrho$ satisfies $\varrho([t,\infty))=L(t)t^{-\alpha}$ for some  $\alpha \in (0,1)$ and some slowly varying function $L(t)$.
	\item[(S$^*$)] The law $\varrho$ satisfies (S) and $\alpha\in (1/2,1)$ or, if $\alpha\leq 1/2$, then the regularity assumptions of \cite[Theorem~1.4]{Caravenna_Doney} are satisfied.
	
\end{enumerate}

\begin{remark}
	It is not straightforward to check for a given measure $\varrho$ if the corresponding renewal process $\cS$ satisfies the gap condition (G). However, in \cite[Proposition~7]{RCP1} it is shown that (G) is satisfied if the tail of $\varrho$ is heavier than $t^{-\alpha}$ for some $\alpha<1$ and some regularity conditions hold  \cite[(A)--(C) in Theorem 1]{RCP1}. Moreover, one can check that every measure $\varrho$ satisfying (S) also satisfies (A)-(C) of \cite{RCP1} and thus (G). Note that condition (S$^*$) implies that $\varrho$ satisfies Erickson's Strong Renewal Theorem \cite[Theorem~1]{Erickson} or its extension \cite[Theorem~1.4]{Caravenna_Doney}.
\end{remark}

Let us briefly discuss the parameter regimes, identifying those that have the potential to lead to new effects that stem from dormancy. 
As already mentioned, we assume that the recovery rate $\delta_d$ of dormant individuals is zero, otherwise we would just obtain a process that can be coupled between two classical contact processes without dormancy in a way that implies a nontrivial phase transition. Moreover, based on the heuristics that most sites are dormant in the long run, we expect this process to have the same behaviour as the usual contact process with infection rate $\lambda_{dd}$ and recovery rate $\delta_d$, not displaying any effect of the heavy tailed dormancy.

\medskip
Of particular interest is the case when some of the infection rates involving at least one dormant individual, i.e.\ $\lambda_{ad}, \lambda_{da}$ or $\lambda_{dd}$, are equal to zero. If all infection rates are positive, the effect of dormancy is essentially the same as having recoveries following a renewal process with heavy tailed interarrival distribution, because then dormant individuals can only recover after waking up, while they spread the infection to all other individuals at a uniformly positive rate. To make this argument precise, one has to consider the effective recovery symbols for $x$, i.e.\ the recovery symbols arriving during an active phase of the site $x$, which we denote by  $\cR_{x}^{\varrho,\delta}$. Clearly,  the set $\cR_{x}^{\varrho,\delta}$ is no renewal process, but a simple computation shows that  $\cR_{x}^{\varrho,\delta}$ also fulfils the gap  condition (G) if $\varrho$ satisfies it. Since in \cite{RCP1} it is shown that the renewal contact process with any positive infection rate survives with positive probability if the set of recovery symbols satisfies (G), we can deduce the same for our process.  Moreover, we expect that the results of \cite{RCP2} (positive critical intensity under conditions that are a bit stronger than having a first moment) carry over easily.

\medskip
In the situation of $\lambda_{dd}=0$ and $\lambda_{aa},\lambda_{ad},\lambda_{da}>0$, clusters of dormant individuals (which will occur sufficiently frequently due to long dormancy periods) will preserve the infection for long times, but still cannot recover. Simulations suggest that in this case we may encounter survival of the process but with only sub-linear growth of the infected region. Indeed, we are able to prove at most logarithmic growth under some conditions:
\begin{theorem}\label{theorem_sublinear_final}
	Let $G=(\Z^d,\E^d)$ be the Euclidean lattice, $\lambda_{dd}=0$ and assume $\varrho$ satisfies condition (S), then the process $r_t$ has at most logarithmic growth, i.e.\ for every $\varepsilon>0$ we have
	\begin{equation}
	\limsup_{t\to\infty}\frac{r_t}{(\ln{t})^{(1+\varepsilon)}}=0\quad\text{almost surely.}\label{log_wachstum}
	\end{equation}
	Moreover, if $\varrho$ satisfies (G) then $r_t$ growth at most sub-linearly, i.e.\ for every 
	$\varepsilon\in(0,\varepsilon^*)$ we have
	\begin{equation}
	\limsup_{t\to\infty}\frac{r_t}{t^{1-\varepsilon}}=0\quad\text{almost surely.}\label{sublinear_growths}
	\end{equation}
	
\end{theorem}
\begin{remark}
	The logarithmic growth in \eqref{log_wachstum} also applies to the Richardson$(\mu)$ model (i.e.\ a contact process without recovery where the infections occur according to a renewal process with interarrival distribution $\mu$) with $\mu$ satisfying (S). In particular, our result improves the at most sub-linear growth given in \cite[Corollary~4.3]{M. Hilario}
\end{remark}
Unfortunately, we are unable to provide survival criteria for the Euclidean lattice. However, on $d$-regular trees, denoted by $\T_d$, a simple coupling with a Galton-Watson process ensures that the process survives if $d\geq 3$.
\begin{proposition}\label{Proposition_survival}
	Let  $G=\T_d$ with $d\geq3$, the parameters $\delta,\sigma,\varrho>0$ be fixed and $\lambda=\lambda_{aa}=\lambda_{ad}$ sufficiently large, then the process survives, i.e.\ $\P(\tau^x=\infty)$ for some $x \in \mathbb T_d$. \color{black}
\end{proposition}

\begin{remark} A more extreme case is $\lambda_{aa}>0,  \lambda_{ad}=\lambda_{da}=\lambda_{dd}=0,$ where dormant individuals neither recover, nor infect anybody or get infected themselves. By our previous reasoning that in the long run most individuals are dormant, we expect the process to die out almost surely since infections are extremely unlikely: The probability that any infected and active individual has an active neighbour at time $t$ goes to 0 as $t \to \infty$ due to the heavy tails of $\varrho$. This is supported by simulations and we refrain from a formal proof. \end{remark}

Another interesting set-up occurs due to the non-Markovianity when $\lambda_{dd}>0$. In this case we can show similar to \cite{RCP_finite} that our process can even survive on finite graphs if a certain cardinality condition on the nodes is satisfied. 

\begin{theorem}\label{theorem:finite_survival_ext}
	Let $G=(V,E)$ be any connected graph and assume $\varrho$ satisfies condition (S$^*$) with some $\alpha\in (0,1)$, then we have for all $\lambda_{dd},\delta,\sigma>0$, $\lambda_{aa},\lambda_{ad},\lambda_{da}\geq 0$ and  $x\in V$ that
	\begin{align}
	\P(\tau^x=\infty)>0 \quad\quad&\text{if}~~|V|>\frac{1}{1-\alpha},\label{card_ass_survival}\\
	\P(\tau^x=\infty)=0 \quad\quad&\text{if}~~|V|<2+\frac{2\alpha-1}{(1-\alpha)(2-\alpha)}.\label{card_ass_ext}
	\end{align}
\end{theorem}
\begin{remark}\label{rem:2.6}
	The survival result also holds in a wider framework, where \emph{general} recovery symbols (which yield a recovery in all phases) and \emph{active} recovery symbols given by renewal processes with interarrival distribution $\mu$ and $\mu_a$, respectively, are allowed.  In this setting, \eqref{card_ass_survival}
	holds for all $\lambda_{dd},\sigma>0$, $\lambda_{aa},\lambda_{ad},\lambda_{da}\geq 0$ and $x\in V$ if $\mu$ and $\varrho$ satisfy condition (S$^*$) for some $\alpha_1$ and $\alpha_2$, respectively, $0<\alpha:=\alpha_1+\alpha_2<1$ and $\mu_a$ fulfills  
	\begin{equation*}\label{cond_a_recovery}
	\P(E^{\mu_a}(t)>s)>0\quad\text{for all }s,t>0.
	\end{equation*}
   However, since the proof of Theorem \ref{theorem:finite_survival_ext} carries over directly and no new arguments are necessary, we omit the proof. 
\end{remark}

\medskip

The remainder of the paper is organized in two additional sections. Section 3 considers the case where $\lambda_{dd}=0$, and Theorem \ref{theorem_sublinear_final} as well as Proposition \ref{Proposition_survival} are proven. Section 4 addresses the case $\lambda_{dd}>0$ and proves Theorem \ref{theorem:finite_survival_ext}.

\section{Logarithmic growth}
We show that the growth of the set of infected vertices is at most logarithmic if $G=(\Z^d,\E^d)$ and $\lambda_{dd}=0$. For notational convenience and simplicity assume $\lambda=\lambda_{aa}=\lambda_{ad}=\lambda_{da}$ and $\delta=0$, thus we have only one type of infection event and no recoveries. 
In \cite[Proposition~1.3~(ii)]{M. Hilario} it was shown via a coupling with an \emph{iterative bond percolation model} that the Richardson($\mu$) model grows sub-linearly if $\mu$ satisfies (G). We will adapt this idea to construct an iterated percolation process $C_n$ which grows approximately linearly and is coupled to $\bfxi_t$ such that $\bfxi_{n(t)}\subset C_{n(t)} $ holds and $n(t)$ grows faster than linear, and even exponentially if $\varrho$ satisfies (S). The specifics of our model however require some care in the construction. In particular, instead of deriving an iterative bond percolation model as an upper bound we couple our model with an \emph{iterated site percolation model} on a new macroscopic grid. The choice of site percolation stems from the fact that it is the activity of the sites (or more precisely the potentially long dormant periods of the sites) that slows down the growth of the infection process. Moreover, we will apply the Dynkin--Lamperti theorem (see e.g. \cite[below~(9.1)]{Erickson}  or 
\cite[Chapter XIV.3]{Feller}), which allows us to improve the result from sub-linear to at most logarithmic growth.
\begin{theorem}[Dynkin--Lamperti theorem]\label{theorem:Dynkin_Lamperti}
	Let $\varrho$ satisfy (S) and  $C_\alpha'=\frac{1}{\Gamma(\alpha)\Gamma(1-\alpha)}$, then
	\begin{equation*}
	\lim_{t\to\infty}\P\bigg(\frac{E^\varrho(t)}{t}\leq x\bigg)=\int_0^x\frac{C_\alpha'}{y^\alpha(y+1)}\d y,\quad \forall x>0.
	\end{equation*}
\end{theorem}
We start by collecting some general results on percolation. For some Bernoulli site percolation $\cP$ on $\Z^d$ with respect to the $\ell_\infty$-norm and parameter $p$ we write $x\leftrightarrow y$ for $x,y\in \Z^d $ if there exists some open path from $x$ to $y$. More precisely, for a collection $\cP=\{B_x\}_{x\in \Z^d}$ of independent Bernoulli random variables with probability $p$ and $x,y\in \Z^d$ we write $x\leftrightarrow y$ if there exists a sequence of sites $x=x_0,x_1,\dots,x_n=y$ such that $||x_i-x_{i-1}||_\infty=1$ and $B_{x_i}=1$ for all $i=1,\dots,n$. Moreover, for $A\subset\Z^d$ we denote by 
\begin{equation*}
    \cC(A):=\bigcup_{x\in A}\{y \colon x\leftrightarrow y\}
\end{equation*}
the union of $A$ with all open clusters under $\cP$ which intersect with $A$ and write $x\leftrightarrow A:=\{\exists y\in A:x\leftrightarrow y\}$. We denote by $\partial B(n)$ the surface of the ball with $\ell_\infty$-radius $n$ centred at the origin and by
\begin{equation*}
    p_c^s:=p_c^{s,\infty}(\Z^d):=\inf\{p \colon \P_p(|\cC(\{0\})|=\infty)>0\}
\end{equation*}
the critical parameter for site percolation. The choice of the $\ell_\infty$-norm is for technical reasons and we will need it below in the grid construction. Then it follows e.g.\ from~\cite[Theorem 1.1]{DCT15} that  for any $p<p_c^{s,\infty}(\Z^d)$ there exists some $\psi(p,d)>0$ such that
	\begin{equation}\label{exp_decay}
	\P_p(0\leftrightarrow\partial B(n))\leq e^{-n\psi(p,d)}\quad\text{for all }n.
	\end{equation}

We now introduce iterated site percolation. For an independent family of Bernoulli site percolation models $\cP_i$ with parameter $p<p_c^s$ and $A\subset \Z^d$  let $\cC_i(A)$ be the union of $A$ and all open clusters under $\cP_i$ which intersect with $A$.
Define $\overline{A}$ as the union of $A$ with its $\ell_\infty$-neighbourhood $\cN(A)=\{z\in \Z^d: \Vert z-x \Vert_\infty=1~\text{for some}~x\in A\}$. For any initial finite $C_0$ we define an increasing sequence of sets by
\begin{equation*}
    C_n=\overline{\cC_n(C_{n-1})}.
\end{equation*}

\begin{lemma}\label{nearly_linear_grow_RN}
	Assume $p<p_c^s(\Z^d)$, $C_0\subset \Z^d$ is finite and let $R_n=\max\{\Vert x \Vert_\infty  \colon x\in C_n\}$, then almost surely
	\begin{equation*}
	    \liminf_{n\to\infty}\frac{R_n}{n}\geq 1\quad\text{and}\quad\limsup_{n\to\infty}\frac{R_n}{n(\ln n)^a}=0\quad\text{for any }a>1.
	\end{equation*}
\end{lemma}
\begin{proof}
	The lower bound is trivial, since in every step we add the whole neighbourhood to the set of nodes. In particular, $R_n\geq R_0+n$ for all $n\geq 1$ which implies the first claim. For the upper bound we follow \cite{M. Hilario} and consider the events
	$$
	A_{n+1}=\{R_{n+1}\geq R_n+\eta \ln R_n\}
	$$
	for some constant $\eta(p,d)>0$. Given the set $C_n$, on the event $A_{n+1}$ there must be some $x\in C_n$ with $\Vert x \Vert_\infty=R_n$ such that there exists a path $x\longleftrightarrow x+ \partial B(\eta \ln R_n-1)$ in $\cP_{n+1}$. Denote by $\cF_n$ the filtration $\cF_n=\sigma(\cP_i:i\leq n)$. Then~\eqref{exp_decay} yields
	\begin{equation*}
	    \P(A_{n+1}|\cF_n)\leq \sum_{x \colon \Vert x \Vert_\infty=R_n}\P(x\longleftrightarrow x+ \partial B(\eta \ln R_n-1)|\cF_n)\leq c R_n^{d-1}e^{-\psi(p,d)(\eta\ln R_n-1)}.
	\end{equation*}
	Choose $\eta(p,d)=\frac{d+1}{\psi(p,d)}$, which gives us $\P(A_{n+1}|\cF_n)\leq c' R_n^{-2}$. Since $R_n\geq n$ we get
	\begin{equation*}
	    \sum_{n\geq 1}\P(A_{n+1}|\cF_n)\leq c' \sum_{n\geq 1}\frac{1}{n^2}<\infty.
	\end{equation*}
	The conditional Borel--Cantelli Lemma (see e.g.\  \cite[Theorem~4.3.4]{Durrett}) ensures that 
	$$
	\P(\limsup_{n\to\infty}A_n)=0,
	$$
	which implies that there exists almost surely some random $n_1$ such that $R_{n+1}\leq R_{n}+\eta \ln R_n$ for $n\geq n_1$. The rest of the proof follows exactly as in \cite{M. Hilario}.
\end{proof}

We now prepare our coupling. To ensure that our infection process is dominated by an iterated percolation process, we construct a time sequence $(t_k)_{k\geq 0}$ and a set-valued discrete time process $(I_k)_{k\geq1}$ with the property that for all $k\geq 1$, from time $t_k$ to $t_{k+1}$ the set $I_k$ is surrounded by a double layer of sites that are dormant throughout the entire time interval. Since a double layer of such dormant sites cannot be traversed by the infection the construction of our sets $(I_k)_{k\geq 1}$ will guarantee that $\bfxi_{t_k}\subset I_k$ holds for all $k\geq 1$ if it holds initially for $k=1$. To make the approach of the double layer rigorous we introduce a macroscopic grid where we merge cubes of $2^d$ sites together and identify them with a node in a new grid $\widetilde{\Z}^d$. Formally, let $A=\{x\in \Z^d:x\geq 0, ||x||_\infty\leq 1\}$, be the cube consisting of $2^d$ sites in the first orthant containing the origin and all nearest neighbours in the first orthant with respect to the $\ell_\infty$-norm.
Moreover, by $A_{(i_1,\dots,i_d)}=A+2(i_1,\dots,i_d)$ we denote the  cube shifted by $2(i_1,\dots,i_d)$. Clearly, these cubes are disjoint, and we can identify every such cube with the site $(i_1,\dots,i_d)$ in the grid $\widetilde{\Z}^d$.
For a sequence of times $(t_k)_{k\geq 0}$ which we define later and $\ell\geq 1$ we say a cube $A_i:=A_{(i_1,\dots,i_d)}$ is $\ell$-bad if
\begin{equation*}
\cS_x\cap[t_{\ell-1},t_{\ell+1}]=\emptyset\quad\text{and}\quad \cU_x\cap[t_{\ell-1},t_{\ell})\neq\emptyset\quad \text{for all }x\in A_{(i_1,\dots,i_d)} ,
\end{equation*}
otherwise we call it $\ell$-good.  By definition, if a cube is $\ell$-bad, all sites of this cube will be dormant during the interval $[t_{\ell},t_{\ell+1}]$. Moreover, since dormant sites can only infect active neighbours, the infection will not traverse any $\ell$-bad cube during the interval $[t_{\ell},t_{\ell+1}]$. 
Hence, we make the following crucial observation: assume at time $t_\ell$ all infected sites are contained in $I_\ell\subseteq \Z^d$ where $I_\ell$ is the union of finitely many cubes. 
Let $I'_{\ell}$  be the union of $I_\ell$ with all $\ell$-good clusters intersecting $I_\ell$ and $I_{\ell+1}=\overline{I'_{\ell}}$ be the union of $I'_{\ell}$ with its neighbourhood. Then all infected sites at time $t_{\ell+1}$ are contained in $I_{\ell+1}$. Note that for this argument the choice of the $\ell_\infty$-norm is crucial.
\begin{proof}[Proof of Theorem \ref{theorem_sublinear_final}]
	What is left to show is the existence of a time sequence $(t_k)_{k\geq 0}$, such that the probability that a cube is $\ell$-good is sufficiently small and the sequence grows sufficiently fast. In particular, the probability that a cube is $\ell$-good should be smaller than the critical value $p_c^{s}>0$ for all $\ell\geq 1$. We start with the case where $\varrho$ satisfies (S). Given some $p'>0$, we can find by Theorem \ref{theorem:Dynkin_Lamperti} some $t_0$ and some constant $c$ such that
	\begin{equation*}
	\P(E^\varrho(t)\leq c t)\leq \frac{p'}{2}+\int_0^c\frac{C_\alpha'}{y^\alpha(y+1)}\d y\leq\frac{p'}{2}+\frac{p'}{2}=p'
	\end{equation*}
	holds for all $t\geq t_0$. Defining 
    $t_k:=(1+c)^{\frac{k}{2}} t_0$ for all $k\geq 1$, we get
	\begin{equation*}
	\P(S_x\cap [t_k,t_{k+2}]=\emptyset)=
    \P(E^\varrho_x(t_k)>c t_k)\geq 1-p'\quad\text{for all }k\geq 1.
	\end{equation*}
	Moreover, by enlarging $t_0$ if necessary, we have
	\begin{equation*}
	\P(\cU_x\cap[t_{k},t_{k+1})\neq\emptyset)=1-\e^{-\sigma(\sqrt{1+c}-1) (1+c)^{k/2}t_0}\geq 1-p'\quad\text{for all }k\geq 1.
	\end{equation*}
	Choosing $p'$ such that $(1-p')^{2^{d+1}}\geq 1-\frac{p_c^{s}}{2}$ yields 
    \begin{equation*}
        \P(A_i\text{ is $\ell$-good})\leq \frac{p_c^{s}}{2}\quad\text{for all }\ell\geq 1~\text{and}~i\in \widetilde{\Z}^d.
    \end{equation*}

    \medskip
	We establish a coupling with an iterated site percolation model as follows: Let $I_1$ be the set of cubes where at least on site is infected at time $t_1$ in the Richardson($\lambda$) model, i.e.\ we ignore all recoveries and the background up to time $t_1$ to avoid time-dependencies. To derive $I_2$, add all cubes to $I_1$ which are neighbours of $I_1$ with respect to the $\ell_\infty$-norm and $1$-good, add also $1$-good neighbours of added neighbours. This gives $I_1'$. Since we almost surely add only finitely many neighbours ($p<p_c^{s}$) the construction on $I_1'$ is well defined. Afterwards set $I_2=\overline{I_1'}$. This step is necessary, since every cube should only be asked once if it is $\ell$-good to avoid arising time-dependencies. Continue the construction for all $k$. Clearly, the so defined process can be coupled with an iterated site percolation model with probability $\frac{p_c^{s}}{2}$ and by our crucial observation from above
	\begin{equation*}
	    \bfxi_{t_k}\subset I_k\subset C_k\quad\text{for all }k\geq 1.
	\end{equation*}
	For some $s\geq t_0$ let $n=n(s)$ be the unique integer such that $s\in[t_{n-1},t_n)$. Since we consider the process without recoveries, $r_s\leq r_{t_n}$ holds, and by our coupling we have $r_{t_n}\leq R_n$. Moreover, Lemma \ref{nearly_linear_grow_RN} gives us that $R_n\leq n^a$ for fixed $a>1$ and sufficiently large $n$. By definition $s\geq t_{n-1}=(1+c)^{\frac{n-1}{2}}t_0$, thus we have $n(s)\leq 2\log_{(1+c)}\big(\frac{s}{t_0}\big)+1$ and conclude
	\begin{equation*}
	\frac{r_s}{(\Tilde{c} \ln(s))^a}\leq 1
	\end{equation*}
	for some constant $\Tilde{c}>0$. Since $a>1$ was arbitrary this finishes the proof of \eqref{log_wachstum}.	

    \medskip
    In case $\varrho$ only satisfies condition (G), we  define our time sequence differently, in fact the same way as it was done in~\cite{M. Hilario} but with a different starting value $t_0$. Let $t^*$ and $\varepsilon^*$ be the constants given in (G).
	For fixed $p'>0$ set  $t_{k+1}:=t_{k}+t_k^{\varepsilon^*}$ with $t_0\geq t^*$ large enough such that for all $k\geq 1$ we have $\P(\cU_x\cap[t_{k},t_{k+1})\neq\emptyset)\geq 1-p'$ and
	\begin{align*}
	\P(\cS_x\cap[t_{k},t_{k+2}]=\emptyset)&=1-\P(\cS_x\cap[t_{k},t_{k+1}]\neq \emptyset\cup \cS_x\cap[t_{k+1},t_{k+2}]\neq \emptyset)\notag\\
	&\geq 1-t_k^{-\varepsilon^*}-t_{k+1}^{-\varepsilon^*}\geq 1-2t_k^{-\varepsilon^*}\geq 1-p'.
	\end{align*}
	As before, $\P(A_i\text{ is $\ell$-good})<\frac{p_c^{s}}{2}$ for all $\ell\geq 1$ if we choose $p'$ small enough. With our coupling from above we get $r_s\leq n(s)^a$ for fixed $a>1$ and $s$ large enough.
	We follow \cite[p.~23]{M. Hilario} to estimate $n(s)\leq \Tilde{c}s^{1-\varepsilon^*}$ and \eqref{sublinear_growths} follows.
\end{proof}  
\begin{proof}[Proof of Proposition \ref{Proposition_survival}]
	Given $\sigma,\delta>0$ we prove that the process survives on $\T_d$ if $\lambda(\sigma,\delta)=\lambda_{aa}=\lambda_{ad}$ is large enough, using a comparison with a Galton--Watson process. It suffices to prove survival for $d=3$, since this assertion implies the same on any copy of $\T_3$ embedded in $\T_d$ if $d>3$.  So assume  w.l.o.g.\ that $d=3$, $\lambda_{dd}=\lambda_{da}=0$ and there is only one type of infection event. Choose some vertex, say $0\in V$, as the root and cut off one sub-tree from it. Let $\bar{V}$ be the remaining tree where all vertices have two children, say $c_1(x)$ and $c_2(x)$, and (apart from the root) one parent vertex, denoted by $p(x)$. For $x\in \bar{V}$ and $t\geq0$ let
	\begin{equation*}
	A(x,t):=\bigg\{\max_{i=1,2}\{E^\lambda_{\{x,c_i(x)\}}(t)\}<\min\{E_x^\delta(t),E_x^\sigma(t)\}\bigg\}
	\end{equation*}
	be the event that from time $t$ onwards there occur two infection events from $x$ to both children before there occurs a recovery or go-to-sleep symbol. Note that by the memoryless property of the exponential waiting times $A(x,t)$ is independent of $\sigma(\bfxi_t)$ and
	\begin{equation*}
	\P(A(x,t))=\P(A(0,0))=\frac{2\lambda}{2\lambda+\sigma+\delta}\frac{\lambda}{\lambda+\sigma+\delta}>\frac{1}{2}\quad\text{for all }x\in \bar{V}, t\geq 0,
	\end{equation*}
	if $\lambda$ is sufficiently large. For convenience set $A(x,\infty):=\emptyset$ and let
	\begin{equation*}
	t^a(x):=\inf\{t\geq 0: \bfxi_t^{0}(x)=(1,a)\}
	\end{equation*}
	be the first time $x$ is infected and active. We define a family of random variables by
	\begin{equation*}
	X_x:=2\cdot\id_{\{A(x,t^a(x)\}} \quad\text{for all }x\in \bar{V}.
	\end{equation*}
    The law of total probability and the independence of  $A(x,t)$ and $\sigma(\bfxi_t)$ give
    \begin{equation*}
        \P(X_x=2|t^a(x)<\infty)
        =\P(A(0,0))\quad \text{for all }x\in \bar{V}.
    \end{equation*}
	The coupling with a Galton--Watson process $(Z_k)_{k\geq 0}$ works as follows: Define recursively a set-valued process $(Y_k)_{k\geq 0}$ with $Y_0=\{0\}$ and 
	$$Y_k:=\{x\in \bar{V}:p(x)\in Y_{k-1},X_{p(x)}=2\}.$$
	By construction it also follows that $x\in Y_k$ implies $t^a(x)<\infty$ and every node $x$ in generation $k$ has $X_x$ offspring, where the $X_x$ are identically distributed and generationwise independent. Since $\E[X_x|x\in Y_k]=2\P(A(x,t))>1$ for all $x\in \bar{V}$ and $k\geq 0$, the process $Z_k:=|Y_k|$ is supercritical and survives with positive probability. On the survival event we have, due to our coupling, $t^a(x)<\infty$ for infinitely many $x\in \bar{V}$. Since our process cannot infect infinitely many vertices in finite time, it follows that $\P(\tau^{0}=\infty)>0$. 
\end{proof}

\section{Finite graphs}
We consider the special case where $G=(V,E)$ is a finite connected graph, $\lambda=\lambda_{dd}>0$  and $\varrho$ satisfies (S$^*$). We will derive two cardinality constraints \eqref{card_ass_survival} and \eqref{card_ass_ext}, which guarantee survival or extinction, respectively, and neither depend on the specific parameters $\lambda,\delta$ or $\sigma$ nor on the exact graph $G$. For the survival subsection \ref{subsection_survival} we can assume w.l.o.g.\ that all other infection rates apart from $\lambda_{dd}$ are zero. For the extinction part \ref{subsection_extinction} we assume for simplicity that $\lambda=\lambda_{dd}=\lambda_{ad}=\lambda_{da}=\lambda_{aa}$.

\subsection{Survival on finite graphs}\label{subsection_survival}

This subsection is based on ideas of \cite{RCP1} and \cite{RCP_finite} to construct a sequence of polynomially increasing time intervals with certain properties. Since we have more possible transitions in our model, we need a finer  subdivision and more careful control on the intervals. More precisely, we want to show that there exists a sequence of polynomially increasing time intervals which have an overlap (subset of the intersection) of sub-polynomial size such that with positive probability:
\begin{enumerate}
	\item In every interval there is one site which does not see any wake-up symbol.
	\item In every overlap there occurs no wake-up symbol for any site.
	\item In the first half of every overlap at all sites there arrives a go-to-sleep symbol.
	\item In the second half of the  overlap there are infection paths from every site to every other site. 
\end{enumerate}
Clearly, the process survives if we find such a sequence. We use the notation and basic construction of~\cite{RCP_finite} as far as possible, and indicate the necessary changes. Given any graph $G=(V,E)$ with $|V|>1/(1-\alpha)$ we choose $\varepsilon>0$ such that $\beta:=|V|(1-\alpha-3\varepsilon)>1$. For $l \geq 1$, a \emph{spanning path of length $l$} in $G$ is a sequence $g=(e_1,\ldots,e_l)$ of edges $e_i=(v_{i-1},v_i) \in E$ such that for any $x,y \in V$ there exist $i\in \{1,\ldots,l\}, j \geq 0$ and a subsequence $e_i,\ldots,e_{i+j}$ of $g$ such that $v_{i-1}=x$ and $v_{i+j}=y$. (I.e.\ the spanning path goes through each vertex at least once, even if it is not necessarily a path in the graph-theoretic sense.) We choose $\gamma>\max\{1,2l/\lambda,2/\sigma\}$, where $l$ describes the minimal length of a spanning path $g$. From now on the spanning path $g$ and $\gamma,\varepsilon>0$ are fixed. For $t>0$ we recursively define a family of random variables $(Y_i^t)_{0\leq i\leq l}$ with $Y_0^t=0$ and 
\begin{equation*}
Y^t_i=Y^t_{i-1}+E^\lambda_{e_i}(Y_{i-1}^t).
\end{equation*}
Note that from time $t$ onward there occur consecutive infection events along the whole spanning path $g$ within time $Y_l^t$. Thus, if some site is infected at time $t$ and all sites are dormant within the time interval $[t,t+Y^t_l]$, then the whole graph is infected at time $t+Y_l^t$. Moreover, for given $t>0$ we define the random variable $X^t$ by
\begin{equation*}
X^t:=\max_{x\in V}E_x^\sigma(t).
\end{equation*}
Thus $X_t$ is the maximum of $|V|$ exponentially distributed random variables with parameter $\sigma$ and describes the time until all sites have seen at least one wake-up symbol (starting from time $t$).
As in \cite{RCP_finite}, we define a polynomially increasing time sequence as follows: For $n\in \N$ let $b_n:=\gamma \log(n)$ and $c_n=b_n^{1+|V|(\alpha+\varepsilon)}$. Clearly, there exists $n_0\in\N$ such that $b_n c_n\leq\frac{n^\varepsilon}{2}$ for all $n\geq n_0$. We set
\begin{equation*}
t_n=\hat{t}_1+\sum_{j=n_0}^n j^\varepsilon-c_jb_j
\end{equation*}
where $\hat{t}_1$ is chosen large enough, such that $\P(E^{\varrho}(t)\leq 1)\leq \frac{1}{t^{1-\alpha-\varepsilon}}$ holds for all  $t\geq \hat t_1$ (which is possible due to (S$^*$), see Proposition 4.1 in \cite{RCP_finite}). Note that $t_{n+1}=t_n+(n+1)^\varepsilon-c_{n+1}b_{n+1}\geq t_n+\frac{(n+1)^\varepsilon}{2}$, $t_{n+1}-t_n\leq (n+1)^\varepsilon$ and $t_n>n$ for $n\geq n_0$. We consider intervals of the form $[t_{n-1},t_{n-1} + n^\varepsilon]$ where the intersection with the next interval is given by $[t_{n},t_{n}+c_{n}b_{n}]$. The overlap will be one of the $c_n$ sub-intervals of this intersection with size $b_n$. We now define four types of events:
\begin{align*}
A_n&:=\{\exists x\in V: E_x^{\varrho}(t_n)>(n+1)^\varepsilon\}\\
B_n&:=\{j\in [0,c_n)\cap \Z: E_x^\varrho(t_n+j b_n)>b_n~\forall x\in V\}\\
C_n&:=\bigcap_{j=0}^{c_n-1}\bigg\{Y_l^{t_n+\frac{2j+1}{2} b_n}\leq \frac{b_n}{2}\bigg\}\\
D_n&:=\bigcap_{j=0}^{c_n-1}\bigg\{X^{t_n+\frac{2j}{2} b_n}\leq \frac{b_n}{2}\bigg\}.
\end{align*}
In words, $A_n$ ensures that there is at least one site seeing no wake-up symbols during the interval $[t_n,t_n+(n+1)^\varepsilon]$. The event $B_n$ makes sure that within the interval $[t_n,t_n+c_n b_n]$ there is at least one sub-interval of size $b_n$ where no site sees any wake-up symbol. The event $D_n$ guarantees that every site sees a go-to-sleep symbol in the first half of each of the $c_n$ sub-intervals of size $b_n$ of $[t_n,t_n+c_n b_n]$. 
Accordingly, $C_n$ describes the event that in every second half of each of these sub-intervals the whole infection path $g$ is traversed by infection events. Note that these infection events can only transmit the infection if both involved sites are dormant. Thus, in contrast to \cite{RCP_finite}, the event $C_n$ alone does not ensure that the whole graph is infected at the end of any sub-interval of size $b_n$ if at least one site was infected initially. Only the events $B_n$, $C_n$ and $D_n$ together guarantee that there exists a sub-interval of size $b_n$ of $[t_n,t_n+c_n b_n]$  such that all sites are dormant during the second half of the interval and the whole graph is infected at the end of the interval if at least one site was infected at the beginning of the second half.

\medskip
Assume that initially one site,  say $0\in V$, is infected and define 
$N_m:=\{E_0^\delta(0)>t_m+c_mb_m\}.$
Clearly, $\P(N_m)=e^{-\delta(t_m+b_m c_m)} >0$ for every fixed $m$. Defining $\hat{A}_m:=\bigcap_{n\geq m} A_n$ for $m\in \N$ and $\hat{B}_m$, $\hat{C}_m$ as well as $\hat{D}_m$ accordingly, it is clear from construction that we have for all $m\in \N$
\begin{equation*}
\{\tau^{0}=\infty\}\supset \hat{A}_m\cap \hat{B}_m\cap\hat{C}_m\cap\hat{D}_m\cap N_m.
\end{equation*}
By Propositions 4.2, 4.3  and 4.5 of \cite{RCP_finite} we find some $m\in \N$ such that $\hat{A}_m$ and $\hat{B}_m$ and $\hat{C}_m$ have positive probability. Moreover, we get $\P(A_n^c)\leq \frac{1}{n^\beta}$ and $\P(B_n^c)\leq \frac{1}{n^\gamma}$ for $n$ sufficiently large.  Note that the proof of $\P(\hat{C}_m)>0$ has to be adapted to take into account that we only have half of the interval size $b_n$ and we changed the definition of $(Y_i^t)_{0\leq i\leq \ell}$. Our choice of $\gamma$ was made such that this still gives the required result. 

\begin{lemma}\label{lemma:infection_stairway}
	There exists some $m\in \N$ such that $\P(\hat{D}_m)>0$.
\end{lemma}
\begin{proof}
	For fixed $t>0$ we have
	\begin{equation*}
	\P(X^t\leq b_n/2)=\P\bigg(\max_{x\in V}E_x^\sigma(t)\leq b_n/2\bigg)=(1-\e^{\frac{-\sigma b_n}{2}})^{|V|}
	\end{equation*}
	and
	\begin{equation*}
	\P(\hat{D}_m)=\P\bigg(\bigcap_{n\geq m}\bigcap_{j=0}^{c_n-1}\{X^{t_n+\frac{2j}{2} b_n}\leq \frac{b_n}{2}\}\bigg)=\prod_{n\geq m}(1-\e^{\frac{-\sigma b_n}{2}})^{|V| c_n}.
	\end{equation*}
	Note that if we fix $\delta>0$ there exists some constant $c>0$ such that $\log(1-x)>-cx$ for all $x\in(0,\delta)$. Therefore, taking the logarithm we obtain, that for some constant $c>0$
	\begin{align*}
	\log\bigg(\prod_{n\geq m}(1-\e^{\frac{-\sigma b_n}{2}})^{|V| c_n}\bigg)&=|V|\sum_{n\geq m}c_n\log(1-\e^{\frac{-\sigma b_n}{2}})
	>-c |V|\sum_{n\geq m}c_n n^{\frac{-\sigma \gamma}{2}}.
	\end{align*}
	By our choice of $\gamma$ we have $\frac{\sigma\gamma}{2}>1$ and by definition of $c_n$ the sum converges, which yields the result.
\end{proof}

\begin{proof}[Proof of Theorem \ref{theorem:finite_survival_ext} part  \eqref{card_ass_survival}]
	Fix some $\delta,\lambda,\sigma>0$. We just have to bound the probability $\P(\hat{A}_m\cap \hat{B}_m\cap\hat{C}_m\cap\hat{D}_m\cap N_m)$ away from zero. Let $m\in \N$ be large enough such that the probability of every single event is bigger than zero. By the independence of the different renewal processes of the construction (Poisson processes included) we have
	\begin{equation*}
	\P(\hat{A}_m\cap \hat{B}_m\cap\hat{C}_m\cap\hat{D}_m\cap N_m)=\P(\hat{A}_m\cap \hat{B}_m) \P(\hat{C}_m)\P(\hat{D}_m)\P(N_m).
	\end{equation*}
 	Observing that
	\begin{equation*}
	1-\P(\hat{A}_m\cap \hat{B}_m)\leq \sum_{n\geq m}\P(A_n^c)+\sum_{n\geq m}\P(B_N^c),
	\end{equation*}
	we may enlarge $m$ such that $\P((\hat{A}_m\cap \hat{B}_m)^c)<1$, and the proof is complete.
\end{proof}

\subsection{Extinction on finite graphs} \label{subsection_extinction}
To show that the condition in \eqref{card_ass_ext} implies extinction we adapt again the ideas from \cite{RCP_finite} and construct a sequence of intervals $([S_{n},S_{n+1}])_{n\geq 1}$ with (random) length $X_n\geq \hat{t}$ such that in every interval there occurs at least one wake-up symbol at time $S_n+X_{n+1,x}$ for every node $x\in V$. Here $\hat{t}  >0$ is a constant from \cite[Proposition~3.2]{RCP_finite}. However, a crucial difference in our model consists in the fact that a site which
wakes up does not have to recover immediately. Therefore we give an argument involving good and bad boxes below. But first, for the sake of completeness, we state the definitions of the random variables $(S_n)_{n\geq 0}$, $(X_n)_{n\geq 1}$ and $(X_{n,x})_{n\geq 1,x\in V}$ according to \cite[p.~1749--50]{RCP_finite}. Let $v_0\in V$ be the initially infected vertex and define 
\begin{equation*}
    X_{1,x}:=\begin{cases}
        E_x^\varrho(0)\quad&\text{if } x=v_0,\\
        0&\text{else.}
    \end{cases}
\end{equation*}
Set $S_1:=X_1:=\max\{X_{1,x}:x\in V\}$, $x_1:=\text{argmax}\{X_{1,x}:x\in V\}$, and $W_{1,x}:=X_1-X_{1,x}$ for all $x\in V$. The rest of the construction works inductively. Assume we have constructed the random variables $X_{m,x}$, $W_{m,x}$, $X_m$, $S_m$ and $x_m$ for all $m=1,\dots, n$ and all $x\in V$, then define
\begin{equation*}
    X_{n+1,x}:=\begin{cases}
        0 \quad&\text{if } x=x_n,\\
        E_x^\varrho(S_n)&\text{if }x\neq x_n \text{ and }W_{n,x}\geq \hat{t},\\
        E_x^\varrho(S_n+\hat{t})+\hat{t}&\text{if }x\neq x_n \text{ and }W_{n,x}<\hat{t}.
    \end{cases}
\end{equation*}
As above, we set $X_{n+1}:=\max\{X_{n+1,x}:x\in V\}$ and $W_{n+1,x}:=X_{n+1}-X_{n+1,x}$ for all $x\in V$, as well as  $S_{n+1}:=S_n+X_{n+1}$ and $x_{n+1}:=\text{argmax}\{X_{n+1,x}:x\in V\}$. A key result of \cite{RCP_finite} is that
\begin{equation}\label{eq:no_infinite_intervalls}
\P(\lim_{n\to\infty} X_n=\infty)=0.
\end{equation}
The proof requires the assumption \eqref{card_ass_ext} and some technical propositions \cite[Propositions 3.1--3.3]{RCP_finite} which we omit here, since no modifications are needed for our model. Given the construction and equation \eqref{eq:no_infinite_intervalls} it is straightforward to show that our process dies out. 
\begin{proof}[Proof of Theorem \ref{theorem:finite_survival_ext} part \eqref{card_ass_ext}]
    Fix some $\delta,\lambda>0$ and some initially infected site $x\in V$. Recall that in this subsection we assume w.l.o.g.\ that all infection rates are $\lambda$. In particular, we have only one type of infection event. Define
	\begin{equation*}
	\Omega^*:=\{\omega \in \Omega \colon \liminf_{n\to\infty}X_n(\omega)<\infty\}\quad\text{and}\quad\Omega^*_m:=\{\omega \in \Omega \colon \liminf_{n\to\infty}X_n(\omega)<m\}
	\end{equation*} for $m\in \N$.  Clearly, $\Omega^*_m\uparrow\Omega^*$ and $\Omega^*_m=\{\omega \colon \exists \{ N_i \colon i\geq 1 \} \subseteq \N \colon N_i\to\infty, X_{N_i}<m \}.$ Let $\varepsilon>0$ be arbitrary and note that by continuity of measures and \eqref{eq:no_infinite_intervalls}, there exists some $\hat{m}$ such that 
	\begin{equation*}
	\P(\Omega^*_m)\geq 1-\varepsilon\quad\text{for all }m\geq \hat{m}.
	\end{equation*}
	Moreover, we have for all $m\in \N$
	\begin{equation*}
	\P(\tau^x=\infty)\leq \P(\{\tau^x=\infty\}\cap \Omega^*_m)+\P((\Omega^*_m)^c).
	\end{equation*}
	We will show that the first probability is zero for every $m\in \N$. Since $\varepsilon>0$ was arbitrary, this suffices to conclude the proof. We say an interval $[S_{n-1},S_{n}]$ is \textit{bad} if:
	\begin{enumerate}
		\item For every wake-up time $X_{n,x}$ with $x\in V$ (these are not necessarily all wake-up times in the interval but $|V|$ many) we have $E_x^{\delta}(X_{n,x})<\min\{1,E_x^\sigma(X_{n,x})\}$. In words: After each of the wake-up symbols $X_{n,x}$ there occurs a recovery symbol at site $x$ within one unit of time and before there arrives a go to sleep symbol.
		\item There occurs no infection event in the whole interval for all sites.
	\end{enumerate}
	An interval which is not bad we call \textit{good}. Clearly, if there occurs one bad interval the infection dies out. Therefore, by construction,
	$\{\tau^x=\infty\}\subset \bigcap_{n\geq 1}\{[S_{n-1},S_n]\text{ is good}\}.$ Observe, given $X_n=S_n-S_{n-1}<m$, we have
	\begin{equation*}
	\P([S_{n-1},S_n]\text{ is bad})\geq \e^{-\lambda m|V|}(1-\e^{-(\delta+\sigma)})^{|V|}\frac{\delta}{\delta+\sigma}=:p_m>0.
	\end{equation*}
	Recapitulate the product structure of our probability space $\Omega=\Omega^\varrho\times\Omega'$ and note that $S_n$ only depends on the renewal processes $\cS_x$ with $x\in V$. Hence, $\Omega_m^*=\Lambda_m\times \Omega'$ for some $\Lambda_m\subseteq \Omega^\varrho$ and for every $\omega_\varrho\in \Lambda_m$ we find (by definition of $\Omega_m^*$ and $S_n$) a sequence  $(N_i(\omega_{\varrho}))_{i\geq 1}$ such that $X_{N_i}=S_{N_{i+1}}-S_{N_i}<m$, $N_{i+1}-N_i>1$ and $S_{N_{i+1}}-S_{N_i+1}>1$. Using the product structure of the measure $\P=\P^\varrho\times \P'$ yields
	\begin{align*}
	\P(\{\tau^x=\infty\}\cap \Omega^*_m)&=\int_{\Lambda_m}\P'(\tau^x(\omega_\varrho,\cdot)=\infty)\d \P^\varrho(\omega_\varrho).
	\end{align*}
	For $\omega_\varrho\in \Lambda_m$, the events $(\{[S_{N_i}(\omega_\varrho),S_{N_i+1}(\omega_\varrho)]\text{ is good}\})_{i\in \N}$ are independent and
	\begin{align*}
	\P'(\tau^x(\omega_\varrho,\cdot)=\infty)&\leq \P'\bigg(\bigcap_{n\in \N}\{[S_{n-1}(\omega_\varrho),S_{n}(\omega_\varrho)]\text{ is good}\}\bigg)\\
	&\leq \prod_{i\in \N}\P'\big(\{[S_{N_i}(\omega_\varrho),S_{N_i+1}(\omega_\varrho)]\text{ is good}\}\big)\leq \prod_{i\in \N}(1-p_m)=0.
	\end{align*}
\end{proof}

\subsection*{Acknowledgements} The authors thank Marco Seiler for interesting discussions and useful comments. We thank two anonymous reviewers for very helpful comments. Funding acknowledgements by the third author: This paper was supported by the János Bolyai Research Scholarship of the Hungarian Academy of Sciences. Project no.\ STARTING 149835 has been implemented with the support provided by the Ministry of Culture and Innovation of Hungary from the National Research, Development and Innovation Fund, financed under the STARTING\_24 funding scheme.

\end{document}